\documentclass[12pt,a4paper]{amsart}
\usepackage{url,color}
\usepackage{amsmath,amsthm,amsfonts,amssymb,latexsym}
\newtheorem{theorem}{Theorem}[section]
\newtheorem{proposition}[theorem]{Proposition}
\newtheorem{lemma}[theorem]{Lemma}
\newtheorem{claim}[theorem]{Claim}

\newtheorem{corollary}[theorem]{Corollary}

\theoremstyle{definition}
\newtheorem{definition}[theorem]{Definition}

\newtheorem{question}[theorem]{Question}

\newcommand{\U}{\mathcal U}
\newcommand{\w}{\omega}

\newcommand{\B}{\mathcal{B}}

\newcommand{\A}{\mathcal{A}}

\newcommand{\CG}{\mathcal{G}}
\newcommand{\CH}{\mathcal{H}}

\newcommand{\F}{\mathcal{F}}

\newcommand{\uhr}{\upharpoonright}

\newcommand{\la}{\langle}
\newcommand{\ra}{\rangle}

\newcommand{\hot}{\mathfrak}

\newcommand{\nothing}[1]{}

\title[$S$-separable spaces]{A density counterpart of the Scheepers covering property}

\author[L. Aurichi, F. Maesano,  L. Zdomskyy]{Leandro Aurichi, Fortunato Maesano,  Lyubomyr Zdomskyy}

\address{Instituto de Ciências  Matemáticas e de Computação, Universidade de São Paulo,
Avenida Trabalhador são-carlense, 400, São Carlos, SP, 13566-590, Brazil}
\email{aurichi@icmc.usp.br}

\address{IIS ``Antonello'', Viale Giostra 2, 98121 Messina (ME), Italy}
\email{fortunato.maesano@unime.it}

\address{Institut f\"ur Diskrete Mathematik und Geometrie,
Technische Universit\"at Wien,
Wiedner Hauptstrasse 8—10/104, 1040 Vienna, Austria}
\email{lzdomsky@gmail.com}
\urladdr{http://www.dmg.tuwien.ac.at/zdomskyy/}


\thanks{The first named author would like to thank the support of Fundação de
Amparo à Pesquisa do Estado de São Paulo (FAPESP) [2023/00595-6]. 
The second author would like to thank the ``National Group for the Algebraic and Geometric Structures and their Applications (GNSAGA--INdAM)'' for generous support for this research. 
The research of the third author was funded by the Austrian Science Fund (FWF) [10.55776/I5930 and 10.55776/PAT5730424].}

\begin{document}

\begin{abstract}
We introduce  a density counterpart of the Scheepers covering property
$\bigcup_{\mathrm{fin}}(\mathcal O,\Omega)$ and study its relations to
 known combinatorial density {properties}. In particular, we show that it is equivalent to the $M$-separability under the Near Coherence of Filters principle of Blass and Weiss.
\end{abstract}

\maketitle

\section{Introduction}
This paper is devoted to combinatorial density properties introduced in 
\cite{Sch99} as counterparts to combinatorial covering properties (also called selection principles), see \cite{coc1,coc2} and references therein. 
 A topological
space $X$ is said \cite{Sch99} to be \emph{$M$-separable}, if for every sequence 
$\la D_n:n\in\w\ra$
of dense subsets of $X$, one can pick finite subsets $F_n\subset D_n$,
$n\in\w$, so that $\bigcup_{n\in\w}F_n$ is dense in $X$.
If we additionally require that 
 every
nonempty open set $U\subset X$ meets all but finitely many $F_n$’s,
then we get the definition of \emph{$H$-separable} spaces introduced in \cite{BelBonMat09}. 
 It is obvious that second-countable spaces  are $H$-separable, and each $H$-separable space is
$M$-separable. 

As indicated above, the $H$- and $M$-separability appeared as counterparts of the classical covering properties of Hurewicz and Menger introduced in \cite{Hur27} and \cite{Men24}, respectively. In \cite{coc1} Scheepers invented
uniform notation for several known combinatorial covering properties 
and created a diagram including these, now named after him,  where in a natural way several new properties appeared in a ``syntactical'' way, i.e., through an application of all known selection procedures to all known kinds of covers, this way making Scheepers diagram more ``complete''.
One of these properties is $\bigcup_{\mathrm{fin}}(\mathcal O,\Omega)$, which is now often called the \emph{Scheepers} property. It appeared to be useful in the study of uniform covering properties of free topological groups \cite{Zdo06} and in several other situations \cite{SzeTsaZdo21}. 

In this paper we suggest a density counterpart of the Scheepers property
and study it by following the patterns from the covering properties. We define a  topological space $X$  to be \emph{$S$-separable}, if for every 
sequence $\la D_n:n\in\w\ra$ of dense subsets of $X$ there exists a sequence
$\la F_n:n\in\w\ra$  such that $F_n\in [D_n]^{<\w}$ for all $n\in\w$,
and for every finite family $\{U_i:i\in k\}$ of open non-empty subsets of
$X$ there exists $n\in\w$ with $U_i\cap F_n\neq\emptyset$ for all $i\in k$.

 First we examine the relation between the $S$-separable and $M$-separable spaces. Since any counterexample for a potential implication can be replaced
 with some of its countable dense subsets, 
 we often {lose} no generality  by reducing our  consideration to countable spaces. 

 If in the definitions of $H$-, $S$- and  $M$-separable spaces we consider only decreasing sequences $\la D_n:n\in\w\ra$ of dense subsets, then we get the definitions of  $mH$-, $mS$-  and $mM$-separable spaces, respectively. Clearly, the ``\emph{m}''-versions are formally weaker, but every $mM$-separable space is actually
 $M$-separable, see \cite{GruSak11}. The situation in the other two cases is more subtle and will be addressed later.

Theorem \ref{main_ncf} below, our first main result,
 is the density counterpart of 
\cite[Corollary~2]{Zdo05}.
As an additional set-theoretic assumption
it uses the following  NCF\footnote{Abbreviated from the \emph{near coherence of filters}.}
principle  introduced by Blass and Weiss in \cite{BlaWei78} and first proved to be consistent in 
\cite{BlaShe87}: 
\begin{quote} For any two non-principal filters $\U_0,\U_1$ on $\w$ there exists a monotone surjection $\phi:\w\to\w$ such that $\phi[\U_0]\cup\phi[\U_1]$ is centered, i.e., $\cap\mathcal C$ is infinite for any finite $\mathcal C\subset \phi[\U_0]\cup\phi[\U_1]$.
\end{quote}
It is known \cite[Theorem~14]{Bla86} that  NCF implies $\hot u<\hot d$ and thus contradicts CH. 

Note that on the covering side we needed a stronger set theoretic assumption:
\cite[Corollary~2]{Zdo05} states that the Scheepers and Menger properties are equivalent
under $\hot u<\hot g$, which is known \cite{BlaLaf89} to imply NCF, while there are models of NCF where $\hot u\geq \hot g$, see \cite{MilShe11}.

\begin{theorem} \label{main_ncf}
(NCF) \ The following conditions are equivalent for 
a countable space $X$:
\begin{enumerate}
\item $X$ is $S$-separable;
\item $X$ is $mS$-separable;
\item $X$ is $M$-separable;
%\item $X$ is $mM$-separable.
\end{enumerate}
\end{theorem}

Theorem~\ref{main_ncf} suggests the following 

\begin{question}
Are the Menger and Scheepers covering properties equivalent under NCF?
\end{question}

Thus, the $mS$-separability is consistently equivalent to 
the $S$-separability. However, this equivalence cannot be established in ZFC, as the following theorem shows. 

\begin{theorem} \label{main_mS_vs_S}
(CH) \ There exists a $mS$-separable but not $S$-separable countable regular space $X$ without isolated points. 
\end{theorem}

Let us note that the space $X$ provided by Theorem~\ref{main_mS_vs_S}
is $M$-separable because it is $mM$-separable being $mS$-separable, and
$mM$-separable spaces are $M$-separable. Thus, CH implies the existence of
an $M$-separable space which is not $S$-separable. On the covering side 
it is known \cite[Theorem~2.1]{SzeTsaZdo21} that such counterexamples can be obtained 
under much weaker assumptions, namely $\hot r\geq \hot d$ suffices 
for the existence of a Menger non-Scheepers set of reals.
This motivated the following 

\begin{question} \label{pr01}
Does $\hot r\geq\hot d$ imply the existence of
a countable $M$-separable space which is not $S$-separable? What about 
$\hot b=\hot d$? What happens in the Laver model?

{Do} any of these assumptions imply the formally stronger statement that there exists an $mS$-separable space
which is not $S$-separable?
\end{question}

\begin{question}
Is it consistent that every countable $mS$-separable space is $S$-separable, but 
there exists a countable $M$-separable space which is not $S$-separable?
\end{question}

We have  mentioned the Laver model in Problem~\ref{pr01}
because it is a classical model of $\hot b=\hot d=\hot r=\hot c$
 without selective  ultrafilters (even $Q$-points), see \cite{Mil80},
while the existence of a selective ultrafilter was crucial for our proof of 
Theorem~\ref{main_mS_vs_S}.

An important class of topological spaces for which the equivalences from 
Theorem~\ref{main_ncf} hold in ZFC is that of so-called $C_p$-spaces:
For a Tychonoff space $T$ we denote by $C_p(T)$ the space 
$\{f:T\to\mathbb R:f$ is continuous $\}$ with the topology inherited from 
the Tychonoff power $\mathbb R^T$.

\begin{theorem} \label{main_cp}
The following conditions are equivalent for 
a Tychonoff space $T$:
\begin{enumerate}
\item $C_p(T)$ is $S$-separable;
\item $C_p(T)$ is $mS$-separable;
\item $C_p(T)$ is $M$-separable;
%\item $X$ is $mM$-separable.
\end{enumerate}
\end{theorem}

An immediate corollary of Theorem~\ref{main_cp},
combined with \cite[Theorems~21 and 40]{BelBonMat09} and \cite[Theorem~8.10]{TsaZdo08}\footnote{In fact, the main result of the unpublished note \cite{ChaPol02} would suffice here instead of \cite[Theorem~8.10]{TsaZdo08}.},  is that there exists a ZFC example of an
$S$-separable non-$H$-separable space. Thus, Theorem~\ref{main_ncf}
has no  ``$H$-separability vs. $S$-separability'' counterpart. 

In light of Theorem~\ref{main_cp} we would like to ask the following

\begin{question} \label{top_group}
Is every $M$-separable (abelian) countable topological group $S$-separable?
What about linear topological spaces over the field of rationals?
\end{question}

By Theorem~\ref{main_ncf} a negative answer to Question~\ref{top_group}
cannot be obtained in ZFC.

Recall that a topological space $X$ is said to be  
\emph{Fr\'echet-Urysohn} (briefly FU) if for every $A\subset X$
and $x\in\bar{A}\setminus A$ there exists a sequence $\la x_n:n\in\w\ra$
of elements of $A$ converging to $x$.

\begin{theorem} \label{fu_vs_s}
  Every {Hausdorff} countable FU  space is $S$-separable.
\end{theorem}

Theorem~\ref{fu_vs_s} can be compared to the  
relation between the FU property and the $H$-separability.
Namely, it follows from \cite[Lemma 2.7(2)]{GruSak11}
 combined with \cite[Corollary
4.2]{GruSak11}  that every countable FU space is $mH$-separable,
while there are ZFC examples of countable  Hausdorff FU spaces
as well as consistent examples of countable regular FU spaces which are 
 not $H$-separable, see \cite[Sections~2 and 3]{BarMaeZdo23}.

This paper is self-contained in the sense that we give definitions
of all notions used in our proofs, unless we find these to be fairly  standard.
On the other hand, {we allow ourselves} to send the reader to the literature we cite
for the definitions of notions used only for explaining the motivation behind this
research etc. For example, we
refer the reader to \cite{coc1,coc2} and \cite{Bla10} for definitions
of combinatorial covering properties and cardinal characteristics  (besides $\hot d$), respectively.

%%%%%%%%%%%%%%%%%%%%%%%%%%%%%%%%%%%%%%%%%%%%%%%%%%%%%%%%%%%%%%%%%%%%%%%%%
%%%%%%%%%%%%%%%%%%%%%%%%%%%%%%%%%%%%%%%%%%%%%%%%%%%%%%%%%%%%%%%%%%%%%%%%%%%
%%%%%%%%%%%%%%%%%%%%%%%%%%%%%%%%%%%%%%%%%%%%%%%%%%%%%%%%%%%%%%%%%%%%%%%%%%

\section{$M$-separable spaces under NCF}

First we introduce free filters on $\w$ as parameters into the notion 
of $S$-separability.

\begin{definition} \label{def0}
Let $\CG$  be a free filter on $\w$.
A space $X$ is called 
\emph{$S_{\CG}$-separable} if for every sequence
$\la D_n:n\in\w\ra$ of dense subsets of $X $ there exists a sequence
$\la F_n:n\in\w\ra$ such that $F_n\in [D_n]^{<\w}$ for all $n\in\w$,
and $\{n\in\w:U\cap F_n\neq\emptyset\}\in\CG$ for all open non-empty $U\subset X$.

If the condition above holds for all decreasing sequences
$\la D_n:n\in\w\ra$ of dense subsets of $X $, then $X$ is called 
\emph{$mS_{\CG}$-separable}.
\end{definition}

Obviously, $S_{\CG}$-separable and $mS_{\CG}$-separable spaces
are $S$-separable and $mS$-separable, respectively.

A subset $\CG_0$ of an ultrafilter $\CG$ is called a \emph{base for}
$\CG$ if for every $G\in\CG$ there exists $G'\in\CG_0$ such that
$G'\subset G$. In this case we say that $\CG $ is \emph{generated} by $\CG_0$.

For a relation $R$ on $\w$ and $x,y\in\w^\w$ we denote by $[x\,R\,y]$
the set $\{n\in\w:x(n)\, R\, y(n)\}$. By definition, $\hot d$
is the minimal cardinality of  $D\subset\w^\w$ such that
for any $x\in\w^\w$ there exists $d\in D$ with $\w\subset [x\leq d]$.

\begin{lemma} \label{less_d}
Let $F\in [\w^\w]^{<\hot d}$ and 
$\mathcal A$ be a family of size less than $\hot d$
such that each $A\in\mathcal A$ is an infinite family of mutually disjoint 
non-empty finite subsets of $\w$. Then there exists 
$h\in\w^\w$ such that 
for every $A\in \mathcal A$ and $f\in F$ there exists \
$a\in A$ such that $a\subset [f<h]$.
\end{lemma}
\begin{proof}
Shrinking each $A\in\A$, if necessary, we may assume that
for any $a_0,a_1\in A$, if there exist $n_0\in a_0$ and $n_1\in a_1$
with $n_0<n_1$, then $\max(a_0)<\min(a_1)$. 
For every $f\in F$ and $A\in \A$ let $f_A\in\w^\w$ be such that 
 $f_A (n)= \max\{f(k):k\in a_n\}$
where $a_n\in A$ is such that $n\leq \max(a_n)$ and there is exactly one
$a'_n \in A$ with $n\leq\max(a'_n)<\max(a_n)$. Let $h\in\w^\w$ be an increasing function such that for every $f\in F$ and $A\in\A$ there exists 
$n\in \w$ with $f_A(n)<h(n)$. Let us fix such $f,A$ and $n$.
Since $n\leq \max(a'_n)<\min(a_n)$,
we have that $n<\min(a_n)$.
Then for every $k\in a_n$ we have 
$$ f(k)\leq f_A(n)={\max\{f(j):j\in a_n\}}<h(n)\leq h(k), $$
which yields $a_n\subset [f<h]$ and thus completes our proof.
\end{proof}

{
Let us note that NCF is equivalent to the statement that
for any non-principal filter $\F_0$ and non-principal ultrafilter $\F_1$ on $\w$, there exists a monotone surjection $\phi:\w\to\w$ with $\phi[\F_0]\subset\phi[\F_1].$ Indeed, since any non-principal filter can be extended to a non-principal ultrafilter, this is obviously a 
formally stronger statement than the NCF. On the other hand,
the NCF as defined in the Introduction would guarantee that 
$\phi[\F_0]\cup\phi[\F_1]$ is centered for some monotone surjection
$\phi:\w\to\w$, which yields $\phi[\F_0]\subset\phi[\F_1]$ if
$\F_1$ (and therefore also $\phi[\F_1]$) is an ultrafilter.}\\

{Recall also that, by \cite[Corollary 5]{Bla86}, NCF implies the existence of free ultrafilters generated by fewer than $\hot d$ sets.}

\begin{proposition}\label{ncf_s_sep}
(NCF) Let $\CG$ be {a free} ultrafilter generated by fewer than  $\hot d$ sets and
$X=\la\w,\tau\ra$ be a countable $M$-separable space. Then $X$ is $S_{\CG}$-separable. 
\end{proposition}
\begin{proof}
Let $\la D_n:n\in\w\ra$ be a sequence of countable dense subsets of 
$X$ and $\tau^*=\tau\setminus\{\emptyset\}$. Since $X$ is $M$-separable, for every $G\in\CG$ there exists an increasing function $f_G\in\w^\w$
such that for every $U\in\tau^*$ the set 
$$H_{G}(U)=\big\{n\in G: U\cap D_n\cap f_G(n)\neq\emptyset\big\}$$
is infinite. 
%(Here the index $\w$ in $H_{f,\w}$ stands for the domain of $f$ and hence %strictly speaking is redundant, but it will be convenient sometimes to keep it.)
%In case $f$ clear from the context we shall simply write $H(U)$ instead of 
%$H_{f}(U)$. 
Given $x\in X$, let us note that the family
$$\CH_{G,x}=\{H_{G}(U):x\in U\in\tau\}$$
is centered
because $H_G(U_0\cap U_1)\subset H_G(U_0)\cap H_G(U_1)$
for any $U_0,U_1\in\tau$ containing $x$.
By NCF there exists a monotone surjection $\phi_{G,x}:G\to\w$ such that
$\phi_{G,x}[\CH_{G,x}]\subset\phi_{G,x}[\CG\uhr G]$,
where $\CG\uhr G=\{G'\cap G:G'\in\CG\}$.

Let $\CG_0\in [\CG]^{<\hot d}$ be a base for $\CG$.
Let also $h\in\w^\w$ be such as in 
 Lemma~\ref{less_d} applied to $F=\{f_G:G\in\CG_0\}$
and 
$$ \A=\big\{\{\phi_{G,x}^{-1}\big(\phi_{G,x}(n)\big):n\in G'\}\: :\: G,G'\in \CG_0, G'\subset G, x\in X\big\}. $$
We claim that 
\begin{equation}\label{pl_1}
 \{n\in\w:U\cap D_n\cap h(n)\neq\emptyset\}\in \CG 
 \end{equation}
for every $U\in\tau^*$, which would complete our proof.

Given any $U\in\tau^*$, fix $x\in U$ and $G\in\CG_0$.
Since $\phi_{G,x}[\CH_{G,x}]\subset\phi_{G,x}[\CG\uhr G]$,
there exists $G_1\in \CG\uhr G$ such that 
$\phi_{G,x}[H_G(U)]=\phi_{G,x}[G_1]$.
Since $\CG_0$ is a base for $\CG$, there exists 
$G'\in\CG_0$  such that $G'\subset G_1$, which yields 
$\phi_{G,x}[G']\subset \phi_{G,x}[G_1]=\phi_{G,x}[H_G(U)]$.
Thus 
$$A:=\{\phi_{G,x}^{-1}\big(\phi_{G,x}(n)\big):n\in G'\}\subset [G]^{<\w}\setminus\{\emptyset\}$$
has the property that $a\cap H_G(U)\neq\emptyset$
for any
$a\in A$, i.e., for any $a\in A$ there exists 
$n_a\in a$ such that $U\cap D_{n_a}\cap f_G(n_a)\neq\emptyset$.
By our choice of $h$ there exists $a\in A$ such that $a\subset [f_G<h]$,
and hence $f_G(n_a)<h(n_a)$. Thus,
$$ U\cap D_{n_a}\cap h(n_a)\neq\emptyset. $$
Since $n_a\in a\subset G$ and $G\in\CG_0$ was chosen arbitrarily,
we have that 
$$ \big\{n\in\w:U\cap D_n\cap h(n)\neq\emptyset\big\}\cap G\neq\emptyset$$
for all $G\in\CG_0$, which proves (\ref{pl_1}) because $\CG$
is an ultrafilter and $\CG_0$ is a base for $\CG$.
\end{proof}

\noindent\textit{Proof of Theorem~\ref{main_ncf}.}
The implication 
$(1)\Rightarrow(2)$ is straightforward, and $(2)\Rightarrow (3)$ holds because $(2)$ implies that $X$ is $mM$-separable, and the $mM$-separability is equivalent
to $M$-separability by \cite[Lemma~2.1]{GruSak11}.

Finally,
$(3)\Rightarrow (1)$ follows from 
Proposition~\ref{ncf_s_sep} since NCF implies that $\hot u<\hot d$, i.e., that there exists an ultrafilter generated by fewer than $\hot d$ sets, see 
\cite[Theorem~14]{Bla86}.
\hfill $\Box$

%%%%%%%%%%%%%%%%%%%%%%%%%%%%%%%%%%%%%%%%%%%%%%%%%%%%%%%%%%%%%%%%%%%
%%%%%%%%%%%%%%%%%%%%%%%%%%%%%%%%%%%%%%%%%%%%%%%%%%%%%%%%%%%%%%%%%%%%%
%%%%%%%%%%%%%%%%%%%%%%%%%%%%%%%%%%%%%%%%%%%%%%%%%%%%%%%%%%%%%%%%%%%%

\section{$mS$-separable  not $S$-separable spaces under CH}

Recall that an ultrafilter $\mathcal G$ on $\w$ is called 
\emph{selective} if for any sequence $\la G_n:n\in\w\ra$
of elements of $\CG$ there exists a number sequence
$\la m_n:n\in\w\ra$ such that $m_n\in G_n$ for all $n\in\w$
and $\{m_n:n\in\w\}\in\mathcal G$.

Here we shall prove the following strengthening of
Theorem~\ref{main_mS_vs_S}.

\begin{theorem} \label{mS_vs_S}
(CH). There exists an $mS$-separable but not $S$-separable countable regular space $X$ without isolated points. More precisely, for every
selective ultrafilter $\CG$ on $\w$ 
there exists an $mS_{\CG}$-separable but not $S$-separable countable regular space $X$ without isolated points.
\end{theorem}
\begin{proof}
We shall construct $X$ in the form $\la\w,\tau_{\w_1}\ra,$
where the topology $\tau_{\w_1}$ is generated by the union of an increasing sequence $\la\tau_\beta:\beta\in\w_1\ra$
of first-countable  zero-dimensional topologies $\la\tau_\beta:\beta\in\w_1\ra$ without isolated points. {Let $\tau_{<0}$ be such a topology on $\w$ that 
$\la\w,\tau_{<0}\ra$ is homeomorphic to $\mathbb Q$, and $\B_{<0}$ be any countable base for $\tau_{<0}$ consisting of clopen sets such that $\w\in\B_{<0}$ and $\emptyset\not\in\B_{<0}$.
Let us write $\w$ as $\bigsqcup_{k\in\w}E_k$ such that each $E_k$
is dense with respect to $\tau_{<0}$, and let 
$\{\la F^\beta_k:k\in\w\ra:\beta\in\w_1\}$
be an  enumeration of $\prod_{k\in\w}[E_k]^{<\w}$ 
such that $F^0_k=\emptyset$ for all $k\in\w$}.
Finally, let $\{\la D'^\beta_n:n\in\w\ra:\beta\in\w_1\}$
be an enumeration of all decreasing sequences of {infinite} subsets of $\w$
such that $\bigcap_{n\in\w}D'^\beta_n =\emptyset$ for all $\beta$, in which each such sequence appears cofinally often
{and $D'^0_n=E_0\setminus n$ for all $n\in\w$}.

{Let $\w=I^0_0\sqcup I^0_1$ be any decomposition into two infinite subsets. For every $n\in\w$ let $E_n=E^0_n\sqcup E^1_n$ be a decomposition of $E_n$ into two dense subsets with respect to $\tau_0$.
Set $U^0_0=\bigcup_{n\in\w}E^0_n, U^0_1= \bigcup_{n\in\w}E^1_n$, $D^0_n=D'^0_n$ and pick $L^0_n\in [D^0_n]^{<\w}$ such that 
$\lim_{n\to\infty}|L^0_n\cap U\cap U^0_s|=\infty$ for all $s\in 2$ and $U\in\mathcal B_{<0}$. Then $\mathcal B_0:= \mathcal{B}_{< 0} \cup \{U^0_s\cap U:U\in\mathcal B_{<0} , s \in 2\}$ is a countable base of clopen sets for a regular topology $\tau_0$ on $\w$ without isolated points in which each $E_n$ is dense. We also set
$G_{0,W}=\w$ for any $W\in\B_0$.}

Suppose that for some $\delta\in\w_1$, $\delta>0$ and all $\beta<\delta$ we have defined
$\tau_\beta$  and a decreasing sequence $\la D^\beta_n:n\in\w\ra$
of {infinite} subsets of $\w$ with empty intersection such that 
\begin{itemize}
\item[$(1)$] $\tau_\beta\subset\tau_\gamma$ for all $\beta\leq\gamma<\delta$;
\item[$(2)$] $\tau_\beta$ is a zero-dimensional topology
generated by a countable clopen base $\B_\beta$ such that $\w\in\B_\beta$, 
 $\emptyset\not\in\B_\beta$ {and $\B_\beta\subset\B_\gamma$ for all
 $\beta\leq\gamma<\delta$};
\item[$(3)$] For every $\beta<\delta$ there exists a sequence $\la L^\beta_n:n\in\w\ra$ such that 
\smallskip

   \begin{itemize}
   \item[$(a)$] $L^\beta_n\in [D^\beta_n]^{<\w}$ for all $n\in\w$;
      \item[$(b)$] For every  $U\in\B_{<\delta}:=\bigcup_{\beta<\delta}\B_\beta $
   there exists $G=G_{\beta,U}\in\CG$ such that
   $$\lim_{n\to\infty,n\in G}|L^\beta_n\cap U|=\infty;$$
   \item[$(c)$] Moreover, for every $U\in\B_{<\delta}$, if 
   $|\{k\in\w:D^\beta_n\cap U\cap E_k\neq\emptyset\}|=\w$
      for all $n\in\w$, {then 
   $$\lim_{n\to\infty,n\in G}|\{k\in\w: L^\beta_n\cap U\cap E_k\neq\emptyset\}|=\infty.$$} In this case for all $n\in G_{\beta,U}$ we set\\ $K^{\beta,U}_n=\{k\in\w: L^\beta_n\cap U\cap E_k\neq\emptyset\}$;
    %\item[$(a)$]
    \end{itemize}
\item[$(4)$] For every $\beta<\delta$ there exist $U^\beta_0, U^\beta_1\in\B_\beta$ which are dense in $\la\w,\tau_{<\beta}\ra$, and a decomposition $\w=I^\beta_0\sqcup I^\beta_1$ such that $\w=U^\beta_0\sqcup U^\beta_1$,
$U^\beta_0\cap\bigcup_{k\in I^\beta_1}F^\beta_k=\emptyset$ and
$U^\beta_1\cap\bigcup_{k\in I^\beta_0}F^\beta_k=\emptyset$.
The topology $\tau_\beta$ is generated by $\B_{<\beta}\cup\{U^\beta_0,U^\beta_1\}$ as a subbase,
and $\B_\beta=\B_{<\beta}\cup\{U\cap U^\beta_s:U\in\B_{<\beta}, s\in 2\}$;
\item[$(5)$] $E_k$ is dense in $\la \w,\tau_{<\delta}\ra$
for all $k\in\w$, {where $\tau_{<\delta}$ is the topology on $\w$ generated by $\B_{<\delta}$}.
\end{itemize}

If $D'^\delta_n$ is dense 
with respect to the topology $\tau_{<\delta}$
generated by $\B_{<\delta}$ for all $n$, then we set 
$D^\delta_n=D'^\delta_n$, otherwise we set $D^\delta_n=\bigcup_{k\geq n}E_k$ for all
$n$. This way $(5)$ guarantees that $D^\delta_n$ is dense in 
$\la \w,\tau_{<\delta}\ra$
for all $n\in\w$. 
Since $\B_{<\delta}$ is countable, we can pick a sequence $\la L^\delta_n:n\in\w\ra$ such that 
   \begin{itemize}
      \item[$(a_\delta)$] $L^\delta_n\in [D^\delta_n]^{<\w}$ for all $n\in\w$;
     \item[$(b_\delta)$] $\lim_{n\to\infty}|L^\delta_n\cap U|=\infty$
   for every $U\in\B_{<\delta}$;    
      \item[$(c_\delta)$] Moreover, for every  $U\in\B_{<\delta}$, if $|\{k\in\w:D^\delta_n\cap U\cap E_k\neq\emptyset\}|=\w$
      for all $n\in\w$, then 
   $$\lim_{n\to\infty}|\{k\in\w: L^\delta_n\cap U\cap E_k\neq\emptyset\}|=\infty.$$ In this case we set $K^{\delta,U}_n=\{k\in\w: L^\delta_n\cap U\cap E_k\neq\emptyset\}$ and $G_{\delta,U}=\w$.
\end{itemize}
Finally, we will construct $U^\delta_0,U^\delta_1,I^\delta_0,I^\delta_1$
such that $(1)$-$(5)$ are satisfied when $\delta$ is replaced with  $\delta+1$.  Let $G'_\delta\in\CG$ be such that 
      $G'_{\delta}\subset^* G_{\beta, U}$ for any\footnote{{Recall that for subsets $A,B$ of $\w$, $A\subset^*B$ means $|A\setminus B|<\w$. }} $\la\beta, U\ra\in\delta\times\B_{<\delta}$.

{
Let $\{J_w,J_h\}$ be a decomposition\footnote{Here ``h'' and ``w'' come from ``height'' and ``width'', respectively.} of $\w$ into two infinite disjoint parts. 
If there exists 
$\la \beta,U\ra\in (\delta+1)\times\B_{<\delta}$ which  is either such as in $(3)(c)$ or $\beta=\delta$
 and $U$ satisfies $(c_\delta)$,  then 
 let $\{\la \beta_j,U_j\ra:j\in J_w\}$ be a (not necessarily injective) enumeration  of such pairs $\la \beta,U\ra$.  
 In other words,
$\la\beta,U\ra=\la\beta_j,U_j\ra$ for some $j\in J_w$
 iff 
$|\{k\in\w:D^\beta_n\cap U\cap E_k\neq\emptyset\}|=\w$
      for all $n\in\w$. If there are no such  $\la\beta,U\ra$, then 
      $\la\beta_j,U_j\ra$ remain undefined for $j\in J_w$.}

{Let us note that by our convention at the beginning of the proof, $\la 0, U^0_s\ra\neq \la\beta_j,U_j\ra$ for any $j\in J_w$, where $s\in 2$.  Let $\{\la \beta_j,U_j\ra:j\in J_h\}$ be a (not necessarily injective) enumeration  of such pairs $\la \beta,U\ra\in (\delta+1)\times\B_{<\delta} $ which do not appear among $\{\la \beta_j,U_j\ra:j\in J_w\}$.  
 In other words,
$\la\beta,U\ra=\la\beta_j,U_j\ra$ for some $j\in J_h$
 iff there exist $n,l\in\w$ such that 
$D^\beta_n\cap U\cap \bigcup_{k\geq l}E_k =\emptyset. $
 Thus, $\la \beta_j,U_j\ra$ are defined for all $j\in J_h$.}

{If 
      $\la\beta_j,U_j\ra$ remain undefined for all $j\in J_w$. then we set 
      $U^{w,\delta}_s=\emptyset$ for all $s\in 2$. Otherwise we 
      shall define $U^{w,\delta}_0$ and $U^{w,\delta}_0$ as follows.}
%%%%%%%%%%%%%%%%%%%%%%%%%%%%%%%%%%%%%%%%%%%%%%%%%%%
%%%%%%%%%%%%%%%%%%%%%%%%%%%%%%%%%%%%%%%%%%%%%%%%%
Set $k_0=0$ and let $\la k_i:i\geq 1\ra$ be a strictly increasing number sequence  such that
\begin{itemize}
\item[$(6)$]
For any $j\in {J}_w$  there exists $i_*(j)\in\w$, $i_*(j)\geq j$, such that
      \begin{itemize}
     \item[$(a)$] $G'_{\delta}\setminus k_{i_*(j)}\subset G_{\beta_j,U_j}$;
     \item[$(b)$] $|K^{\beta_j,U_j}_n\setminus k_i|\geq i^2$ for all 
     $i\geq i_*(j)$ and  $n\in G'_{\delta}\setminus k_{i+1}$;
     \item[$(c)$] $K^{\beta_j,U_j}_n\subset k_{i+1}$ for all $i\geq i_*(j)$ and $n\in
     G'_{\delta}\cap [k_{i_*(j)},k_i)$.
     \end{itemize}
\end{itemize}
Without loss of generality we can assume that 
$T=\bigcup_{i\in\w}[k_{3i+1}, k_{3i+2})\in\CG$. Let $G_\delta\subset T\cap G'_\delta$,
$G_\delta\in\CG$ be such that $|G_\delta\cap[k_{3i+1},k_{3i+2})|\leq 1$ for all $i\in\w$.
We are now in a position to
construct $I^\delta_0,I^\delta_1$ as well as 
the first ``third'' of 
$U^\delta_0,U^\delta_1$
in the form of unions 
$I^{\delta}_0=\bigcup_{i\in\w}I^{\delta,i}_0$,
$I^{\delta}_1=\bigcup_{i\in\w}I^{\delta,i}_1$,
$U^{w,\delta}_0=\bigcup_{i\in\w}U^{w,\delta,i}_0$ and
$U^{w,\delta}_1=\bigcup_{i\in\w}U^{w,\delta,i}_1$ such that
\begin{itemize}
\item[$(7)$] $I^{\delta,i}_s\subset [k_{3i},k_{3i+3} ) $,
where $s\in 2$;
\item[$(8)$] $U^{w,\delta,i}_s\in \big[\bigcup_{k\in I^{\delta,i}_s  }E_k \big]^{<\w}$,
where $s\in 2$;
\item[$(9)$]  $I^{\delta,i}_0\cap I^{\delta,i}_1=\emptyset$ 
(and hence also $U^{w,\delta,i}_0\cap U^{w,\delta,i}_1=\emptyset$)
for all $i\in\w$.
\end{itemize}
\addtocounter{equation}{9}
Set $I^{\delta,i}_0= I^{\delta,i}_1 =U^{w,\delta,i}_0= U^{w,\delta,i}_1=\emptyset$
if $i=0$ or $G_\delta\cap [k_{3i+1},k_{3i+2})=\emptyset$.
Given $i>0$ such that there exists (a necessarily unique)
$n_i\in G_\delta\cap[k_{3i+1},k_{3i+2})$, set $J^\delta_i=\{j\in i\cap J_w: i_*(j)\leq 3i\}$. 
%and $J^{\downarrow\delta}_i=\{j\in i: i_*(\beta_j, U_j)\leq 3i\}$.
 Note that   $(6)$ 
  yields
 $$ \big|K^{\beta_j,U_j}_{n_i}\cap [k_{3i},k_{3i+3})\big|\geq (3i)^2   $$
 for all $j\in J^\delta_i$.
Since $|J^\delta_i|\leq i$,
we can pick $I^{\delta,i}_0\subset [k_{3i}, k_{3i+3})$
of size $|I^{\delta,i}_0|= i^2$ such that
$|I^{\delta,i}_0\cap K|\geq i $ for all
$K\in \big\{K^{\beta_j, U_j}_{n_i} : j\in J^{\delta}_i \big\}.$
Since all such $K$ have  $\geq 9i^2$ elements in $[k_{3i}, k_{3i+3})$,
we conclude that for  
 $I^{\delta,i}_1:=[k_{3i}, k_{3i+3})\setminus I^{\delta,i}_0$
 we have $|I^{\delta,i}_1\cap K|\geq 8i^2\geq i $ for all $K$ as above. 
 Then the sets
 $$
  U^{w,\delta,i}_s= \bigcup_{j\in J^{\delta}_i}L^{\beta_j}_{n_i} \cap \bigcup_{k\in I^{\delta,i}_s} E_k,    
 $$
 where $s\in 2$, are mutually disjoint. This completes our definition of 
$I^{\delta,i}_0$, $I^{\delta}_1$, $U^{w,\delta}_0$ and
$U^{w,\delta}_1$. Let us note that 
\begin{equation} \label{w0}
U^{w,\delta}_s\subset\bigcup\big\{E_k:k\in I^\delta_s=\bigcup_{i\in\w }I^{\delta,i}_s \big\}    \mbox{ \ and}
\end{equation}
if
{$j\in J^\delta_i$ and}
$k\in K^{\beta_j,U_j}_{n_i}\cap I^{\delta,i}_s$, then
$L^{\beta_j}_{n_i}\cap E_k\cap U_j\neq \emptyset$ and $L^{\beta_j}_{n_i}\cap E_k\subset U^{w,\delta,i}_s$, which yields
$$U^{w,\delta,i}_s\cap L^{\beta_j}_{n_i}\cap E_k\cap U_j\neq \emptyset,$$
where $s\in 2$.
Consequently, 
\begin{eqnarray}\label{w1}
|\{k\in\w: U^{w,\delta}_s\cap L^{\beta_j}_{n_i}\cap E_k\cap U_j\neq 
\emptyset \}|\geq |K^{\beta_j,U_j}_{n_i}\cap I^{\delta,i}_s|\geq i 
\end{eqnarray}
for all $i$ such that $n_i\in G_\delta$,  $j\in J^\delta_i$ and $s\in 2$.

%%%%%%%%%%%%%%%%%%%%%%%%%%%%%%%%%%%%%%%%%%%%%%%%%%%%%%%%%%%%%%%%%%%%%%%%%%%%%%%%%%%
%%%%%%%%%%%%%%%%%%%%%%%%%%%%%%%%%%%%%%%%%%%%%%%%%%%%%%%%%%
%%

Next, we shall work with pairs $\la \beta_j, U_j\ra$ for $j\in J_h$, i.e.,
that are pairs   $\la \beta_j, U_j\ra$ for which
there exist $n=n(j)\in\w$ and $l=l(j)\in\w$
such that $D^{\beta_j}_n\cap U_j\subset\bigcup_{k\in l}E_{k}$.
Without loss of generality we may assume that
$n(j)\geq n(j')\geq j'$ and  $l(j)\geq l(j')\geq j'$ for all
$j\geq j'$ in $J_h$.

\begin{claim} \label{cl01}
There exists a subset   $M_\delta =\{m_j: j\in J_h\}\in\CG$ of $G_\delta$ such that the following conditions hold for all $j\in J_h$:
\begin{itemize}
\item[$(i)$] $m_{j}>m_{j'}$ for all $j>j'$ in $J_h$;
\item[$(ii)$] 
$c_j:=\min\big(\bigcup_{j'\in(j+1)\cap J_h} D^{\beta_{j'}}_{m_j}  \big)> a_j:=
\max\big(\bigcup_{j'',j^{(3)}\in J_h\cap j} L^{\beta_{j''}}_{m_{j^{(3)}}}   \big)$,  and  \\ hence 
$D^{\beta_{j'}}_{m_j}\cap L^{\beta_{j''}}_{m_{j^{(3)}}}=\emptyset$
for any $j'\in J_h\cap (j+1)$ and $j'',j^{(3)}\in J_h\cap j$;
\item[$(iii)$] $c_j>b_j:=\max\Big(\big (\bigcup_{k\in \w}F^\delta_k\cup U^{w,\delta}_0\cup U^{w,\delta}_1\big)\cap\bigcup_{k\in l(j)}E_k\Big)$,
and \\ hence
$\bigcup_{j'\in J_h\cap (j+1)} (D^{\beta_{j'}}_{m_j}\cap U_{j'})\cap 
\big (\bigcup_{k\in \w}F^\delta_k\cup U^{w,\delta}_0\cup U^{w,\delta}_1\big)=\emptyset$;
\item[$(iv)$] 
$\Big|\big(\max\{a_j,b_j\},c_j\big)\cap E_{p}\cap U_{q}\Big|> j^2$ for all $p,q\in j$;
\item[$(v)$] $M_\delta\setminus m_j\subset G_{\beta_j,U_j} $;
\item[$(vi)$]  $|L^{\beta_{j'}}_{m_j}\cap U_{j'}|>j^2+j$
for all $j'\in (j+1)\cap J_h$.
\end{itemize}
\end{claim}
\begin{proof}
We shall select $m_j$ in a course of the following game of length $\w$,
whose innings are indexed by elements of $J_h$:
in the inning number $j\in J_h$ player I chooses $M_j\in \CG$, and player II replies by
selecting $m_j\in M_j$. Player II wins if $\{m_j:j\in J_h\}\in\CG$. It is known \cite[Theorem~2.6]{Laf96} that $I$ has no winning strategy in this game because 
$\CG$ is selective. 

Next, we shall define a strategy  for player I in the game described above.
Letting $j_{\min}=\min J_h$,
 I starts with 
$M_{j_{\min}}=G_\delta$, and II replies by choosing $m_{j_{\min}}\in M_{j_{\min}}$
such that conditions $(iii)$-$(vi)$ are satisfied for $j=j_{\min}$.
Regarding $(iii)$, $b_{j_{\min}}$ is well-defined because the set 
$$\big (\bigcup_{k\in \w}F^\delta_k\cup U^{w,\delta}_0\cup U^{w,\delta}_1\big)\cap\bigcup_{k\in l(j_{\min})}E_k$$
is finite since by the construction we 
have that $|U^{w,\delta}_s\cap E_k|<\w$ for all $k\in\w$ and $s\in 2$.
(The ``hence'' part of $(iii)$ follows from the definition of $c_j$ and
the inclusion $\bigcup_{j'\in J_h\cap (j+1)} (D^{\beta_{j'}}_{m_j}\cap U_{j'})\subset \bigcup_{k\in l(j)}E_k$.)

Similarly, in the $j$-th inning for $j>j_{\min}$, $j\in J_h$,
player I starts with 
$M_j=G_\delta$, and II replies by choosing $m_j\in M_j$
such that conditions $(i)$-$(vi)$ are satisfied for $j$.
Regarding $(iii)$, $b_j$ is well-defined 
for the same reasons as in case $j=j_{\min}$.

Since the strategy for player I described above cannot be winning, we get the
desired sequence $\la m_j:j\in J_h\ra$.
\end{proof}

Next, we shall construct the second ``third'' of
$U^\delta_0$ and $U^\delta_1$.
Given $j\in J_h$, let us fix $C_{j',j}\in [L^{\beta_{j'}}_{m_j}\cap U_{j'}]^j$
for all $j'\in (j+1)\cap J_h$ and set 
$U^{h,\delta,j}_0=\bigcup_{j'\in (j+1)\cap J_h}C_{j',j}$.
Thus 
\begin{equation} \label{h0}
|U^{h,\delta,j}_0\cap L^{\beta_{j'}}_{m_j}\cap U_{j'}|\geq j
\end{equation}
for all $j'\in (j+1)\cap J_h$. 
Since $|U^{h,\delta,j}_0|\leq j^2$, for
\begin{equation}
    U^{h,\delta,j}_1:=\bigcup_{j'\in (j+1)\cap J_h}(L^{\beta_{j'}}_{m_j}\cap U_{j'})\setminus U^{h,\delta,j}_0
\end{equation}
it follows from $(vi)$ that 
\begin{equation} \label{h1}
|U^{h,\delta,j}_1\cap L^{\beta_{j'}}_{m_j}\cap U_{j'}|\geq j
\end{equation} 
for all $j'\in (j+1)\cap J_h$. 

Let us note that $U^{h,\delta,j}_s\subset\bigcup_{j'\in J_h\cap (j+1)}L^{\beta_{j'}}_{m_j}\subset \bigcup_{j'\in J_h\cap (j+1)}D^{\beta_{j'}}_{m_j}$ for all $s\in 2$,
and hence $(ii)$ guarantees that 
$U^{h,\delta,j}_s\cap U^{h,\delta,j'}_{s'} =\emptyset$
for any $j'\in j\cap J_h$ and $s,s'\in 2$. 
As a result, the sets $U^{h,\delta}_0:=\bigcup_{j\in J_h} U^{h,\delta,j}_0$
and $U^{h,\delta}_1:=\bigcup_{j\in J_h} U^{h,\delta,j}_1$
are also disjoint.  

%%%%%%%%%%%%%%%%%%%%%%%%%%%%%%%%%%%%%%%%%%%%%%%%%%%%%%%%%%%%%%%%%%%%%%%%%%
It follows from $(i)$ and $(iii)$ that 
$$U^{h,\delta,j}_0\cap
\big (\bigcup_{k\in \w}F^\delta_k\cup U^{w,\delta}_0\cup U^{w,\delta}_1\big)=\emptyset =
U^{h,\delta,j}_1\cap \big (\bigcup_{k\in \w}F^\delta_k\cup U^{w,\delta}_0\cup U^{w,\delta}_1\big) $$
for all $j\in J_h$.
Consequently, 
\begin{equation} \label{disj_F's}
U^{h,\delta}_0\cap
\big (\bigcup_{k\in \w}F^\delta_k\cup U^{w,\delta}_0\cup U^{w,\delta}_1\big)=\emptyset =
U^{h,\delta}_1\cap \big (\bigcup_{k\in \w}F^\delta_k\cup U^{w,\delta}_0\cup U^{w,\delta}_1\big). 
\end{equation}

%%%%%%%%%%%%%%%%%%%%%%%%%%%%%%%%%%%%%%%%%%%%%%%%%%%%%%%%%
The next ``third'' part of $U^\delta_0,U^\delta_1$, namely
$U^{d,\delta}_0$ and $U^{d,\delta}_1$, will guarantee, in particular,
that $(5)$ holds for $\delta+1$ instead of $\delta$.
For every $j\in J_h$ and $p,q\in j$ fix 
$$s^j_{p,q}\in (\max\{a_j,b_j\},c_j)\cap E_p\cap U_q,$$
 set $U^{d,\delta,j}_0=\{s^j_{p,q}:p,q\in j\}$ and
 $U^{d,\delta,j}_1=(\max\{a_j,b_j\},c_j)\setminus U^{d,\delta}_0$.
 By the definition of $U^{d,\delta,j}_0$, $(iv)$ and $|U^{d,\delta,j}_0|\leq j^2$ we have 
 \begin{equation}\label{all_e_dense}
 U^{d,\delta,j}_0\cap E_p\cap U_q\neq \emptyset \mbox{ \ and\  } U^{d,\delta,j}_1\cap E_p\cap U_q\neq \emptyset
 \end{equation}
 for all $p,q\in j$. Let us also note that 
 \begin{equation} \label{no_inter_with_F's}
 U^{d,\delta,j}_s\cap \big (\bigcup_{k\in \w}F^\delta_k\cup U^{w,\delta}_0\cup U^{w,\delta}_1\big)=\emptyset
 \end{equation}
for any $s\in 2$. Indeed, 
\begin{eqnarray*}
U^{d,\delta,j}_s\cap \big (\bigcup_{k\in \w}F^\delta_k\cup U^{w,\delta}_0\cup U^{w,\delta}_1\big)\subset \\
\subset \big((\w\setminus (b_j+1))\cap\bigcup_{p\in j}E_p\big) \cap
(\bigcup_{k\in \w}F^\delta_k\cup U^{w,\delta}_0\cup U^{w,\delta}_1\big)\subset \\
\subset (\w\setminus (b_j+1))\cap  
\Big(\big (\bigcup_{k\in \w}F^\delta_k\cup U^{w,\delta}_0\cup U^{w,\delta}_1\big)\cap\bigcup_{k\in l(j)}E_k\Big)\subset \\
\subset (\w\setminus (b_j+1))\cap (b_j+1)=\emptyset.
\end{eqnarray*}
Letting $U^{d,\delta}_s=\bigcup_{j\in J_h}U^{d,\delta,j}_s$
for $s\in 2$, we conclude from (\ref{no_inter_with_F's})
that 
\begin{equation} \label{no_inter_with_F's_1}
 U^{d,\delta}_s\cap \big (\bigcup_{k\in \w}F^\delta_k\cup U^{w,\delta}_0\cup U^{w,\delta}_1\big)=\emptyset
 \end{equation}
for any $s\in 2$. 

Conditions $(ii)$ and $(iii)$ yield
$a_j>c_{j'}>a_{j'}$
for all $j'<j$ in $J_h$, and  by the definition $U^{h,\delta,j}_s\subset [c_j,a_{\min(J_h\setminus(j+1))}]$ for all
$j\in J_h$ and $s\in 2$, and
hence 
\begin{equation} \label{no_inter_with_h's}
 U^{d,\delta}_s\cap \big(U^{h,\delta}_0\cup U^{h,\delta}_1\big)=\emptyset
 \end{equation}
for all $s\in 2$. Summarizing the above we get that 
the families
$$\{U^{w,\delta}_0,U^{w,\delta}_1, U^{h,\delta}_0,U^{h,\delta}_1, U^{d,\delta}_0,U^{d,\delta}_1\} \mbox{ \ \ \ and}$$
$$\{\bigcup_{k\in\w}F^\delta_k, U^{h,\delta}_0,U^{h,\delta}_1, U^{d,\delta}_0,U^{d,\delta}_1\} $$
consist of mutually disjoint elements. 

%%%%%%%%%%%%%%%%%%%%%%%%%%%%%%%%%%%%%%%%%%%%%%%%%%%%%%%

Finally, we set
$$U^{\delta}_0=U^{w,\delta}_0\cup U^{h,\delta}_0\cup U^{d,\delta}_0\cup\big(\bigcup_{k\in I^\delta_0}E_k \: \setminus\: 
(U^{w,\delta}_1\cup U^{h,\delta}_1\cup U^{d,\delta}_1)\big),$$
$$U^{\delta}_1=U^{w,\delta}_1\cup U^{h,\delta}_1\cup U^{d,\delta}_1\cup \big(\bigcup_{k\in I^\delta_1}E_k \: \setminus\: 
(U^{w,\delta}_0\cup U^{h,\delta}_0 \cup U^{d,\delta}_0)\big),$$
and note that 
$U^\delta_s\supset U^{w,\delta}_s\cup U^{h,\delta}_s\cup U^{d,\delta}_s$ for all $s\in 2$
as well as $\w=U^\delta_0\cup U^\delta_1$ because
$\bigcup_{k\in I^\delta_s}E_k$ is easily seen to be a subset of
$U^\delta_0\cup U^\delta_1$ for all $s\in 2$. From (\ref{w0}),  (\ref{disj_F's})
and (\ref{no_inter_with_F's_1}) it follows immediately that 
\begin{equation}\label{non-S-witness}
U^\delta_s\cap \bigcup_{k\in I^\delta_{1-s}}F_k=\emptyset
\end{equation} 
for all $s\in 2$. Thus $I^\delta_0, I^\delta_1,U^\delta_0, U^\delta_1$,
$$\B_\delta=\B_{<\delta}\cup\{U\cap U^\delta_s:U\in\B_{<\delta}, s\in 2\}$$
and the topology $\tau_\delta$ generated by $\B_\delta$
satisfy $(4) $ when $\beta$ is replaced with $\delta$, the density 
of $U^\delta_0, U^\delta_1$ in $\la\w,\tau_{<\delta}\ra$ being a consequence of
(\ref{all_e_dense}).
Conditions $(1)$ and $(2)$ are also clearly satisfied, while 
$(5)$ for $\delta+1$ 
(note that $\tau_{<\delta+1}=\tau_\delta$)  immediately follows from (\ref{all_e_dense}) since it yields
 $U^\delta_s\cap E_q\cap U_p\neq\emptyset$
for all $p,q\in\w$, i.e., any element of $\B_\delta$ intersects
all $E_q$'s. In order to verify $(3)$ for $\delta+1$
we have to consider 2 cases.

I)\;  $\beta<\delta+1$, $U\in\B_\delta\setminus\B_{<\delta}$.
Then there exists $j\in\w$ and $s\in 2$ such that
$\beta=\beta_j$ and $U=U^\delta_s\cap U_j$. Here two subcases are possible.

a)\;  $j\in J_w$. Then $G_{\beta,U}:=G_\delta\setminus 3i_*(j)$
is as required. Indeed, each $n\in G_{\beta,U}$ is of the form
$n_i$ for some $i\geq i_*(j)$ by the definition of $G_\delta$.
Applying (\ref{w1}) we conclude that  
$$|\{k\in\w:U\cap L^\beta_{n}\cap E_k\neq\emptyset\}|\geq |\{k\in\w: 
U^{w,\delta}_s\cap U_j\cap L^{\beta_j}_{n_i}\cap E_k\neq\emptyset\}|\geq i,$$
and hence 
$$\lim_{n\to\infty,n\in G_{\beta,U}}|\{k\in\w:U\cap L^\beta_{n}\cap E_k\neq\emptyset\}|=\infty.$$

b)\; $j\in J_h$. In this case 
$$U\cap D^\beta_{n(j)}=U^\delta_s\cap U_j\cap D^{\beta_j}_{n(j)}\subset
U_j\cap D^{\beta_j}_{n(j)}\subset\bigcup_{k\in l(j)}E_k,$$
and hence we need to prove only $(3)(b)$ in this case.
We claim that 
$G_{\beta,U}:=\{m_{v}:v\geq j\}$
is as required. Indeed, if
$v\geq j$  then 
$$|U\cap L^\beta_{m_v}|=|U^\delta_s\cap L^{\beta_j}_{m_v}\cap U_j|\geq 
|U^{h,\delta,v}_s\cap  L^{\beta_j}_{m_v}\cap U_j|\geq v$$
by (\ref{h0}) and $(\ref{h1})$, which yields 
$\lim_{n\to\infty,n\in G_{\beta,U}}|U\cap L^\beta_{n}|=\infty.$

II)\;  $\beta=\delta$ and $U\in\B_{<\delta}$. Then $(3)$ is satisfied 
for $G_{\beta,U}=\w$ by $(b_\delta)$ and $(c_\delta)$.

This completes our recursive construction of $\tau_{\w_1}$.
The space $X=\la\w,\tau_{\w_1}\ra$ is not $S$-separable
by $(4)$. To show that $X$ is $mS$-separable let us fix 
a decreasing sequence
$\la D_n:n\in\w\ra$ of dense subsets of $X$ with $\bigcap_{n\in\w}D_n=\emptyset$
and find $\beta$ with $D'^\beta_n=D_n$ for all $n\in\w$.
Since each $D_n$
 is dense in $X$, it is also dense in $\la\w,\tau_{<\beta}\ra$,
 and therefore $D^\beta_n=D_n=D'^\beta_n$ for all $n\in\w$.
 Let $U\in\B_{<\w_1}$ and $\delta<\w_1$ be such that $U\in\B_{<\delta}$.
 Then $(3)(b)$ implies that the sequence
 $\la L^\beta_n:n\in\w\ra$ is witnessing that $X$ is 
 $mS_{\CG}$-separable.
\end{proof}

%%%%%%%%%%%%%%%%%%%%%%%%%%%%%%%%%%%%%%%%%%%%%%%%%%%%%%%%%%%%
%%%%%%%%%%%%%%%%%%%%%%%%%%%%%%%%%%%%%%%%%%%%%%%%%%%%%%%%%%%%
%%%%%%%%%%%%%%%%%%%%%%%%%%%%%%%%%%%%%%%%%%%%%%%%%%%%%%%%%%%%

\section{Spaces of functions and FU spaces}

First we shall show that the $S$-separability is closely related to the weak
form of $M$-separability of finite powers defined below.

\begin{definition}
 A topological space  $X$ is \emph{$pM$-separable\footnote{``p'' comes here from ``powers''.},}  if for every sequence $\la D_k:k\in\w\ra$ of dense subsets of 
 $X$ and $n\in\w$, there exists a sequence $\la F_k:k\in\w\ra$ of finite subsets 
 of $X$ such that $\bigcup_{k\in\w} (F_k^n\cap D_k^n)$ is dense in $X^n$. \hfill $\Box$
\end{definition}

Clearly, if $X^n$ is $M$-separable for all $n\in\w$, then $X$ is $pM$-separable.
%The following fact is a density counterpart of \cite[????]{????}.

\begin{proposition} \label{pM_vs_S}
A space $X$ is $S$-separable iff it is $pM$-separable. 
\end{proposition}
\begin{proof}
Suppose that $X$ is $S$-separable and fix $n\in\w$ and a sequence
$\la D_k:k\in\w\ra$ of dense subsets of $X$. Let $\la F_k:k\in\w\ra$
be a witness for the $S$-separability of $X$, where $F_k\in [D_k]^{<\w}$ for all $k\in\w$. We claim that $\bigcup_{k\in\w} F_k^n$ is dense in $X^n$.
Indeed, let $\emptyset\neq W\subset X^n$ be open and $\la U_i:i\in n\ra$  be a sequence of length $n$ of open non-empty subsets of $X$ such that $\prod_{i\in n}U_i\subset W$.
The choice of $F_k$'s yields $k\in\w$ such that $F_k\cap U_i\neq\emptyset$
for all $i\in n$. Then $F_k^n\cap\prod_{i\in n}U_i\neq\emptyset, $
and hence also $F_k^n\cap W\neq\emptyset$.

Now assume that $X$ is $pM$-separable and fix a sequence
$\la D_k:k\in\w\ra$ of dense subsets of $X$. It follows that for every 
$n\in\w$ we can find a sequence $\la F_{n,k}:k\in\w\ra$ such that $F_{n,k}\in [D_k]^{<\w}$, and for every open  $W\subset X^n$ the set 
$$\{k\in\w:W\cap F_{n,k}^n\neq\emptyset\}$$
is infinite. We claim that $\la F_k:k\in\w\ra$ such that $F_k=\bigcup_{n\leq k}F_{n,k}$ is witnessing the $S$-separability of $X$. Indeed, let $n\in\w$ and $\la U_i:i\in n\ra$ be a sequence of $n$ non-empty open subsets of $X$.
By the choice of   $\la F_{n,k}:k\in\w\ra$ we can find 
$k\geq n$ such that $F_{n,k}^n\cap\prod_{i
\in n}U_i\neq\emptyset$. Thus, $F_{n,k}\cap U_i\neq\emptyset $ for all $i\in n$.
Since $n\leq k$ we have $F_{n,k}\subset F_k$, and therefore $F_{k}\cap U_i\neq\emptyset $ for all $i\in n$, which completes our proof. 
\end{proof}

%The following fact is a density counterpart of \cite[????]{????}.

\begin{corollary}\label{cor_sfinomom_ufinomon}
If $X^n$ is $M$-separable for every $n\in\w$, then $X$
is $S$-separable.
\end{corollary}

We are in a position now to present the 
\medskip 

\noindent\textit{Proof of Theorem~\ref{main_cp}.}
As in Theorem~\ref{main_ncf} it suffices to prove the implication $(3)\Rightarrow (1)$. But it is known \cite[Corollary 2.12]{BelBonMatTka08} that the $M$-separability of $C_p(T)$
implies that $C_p(T)^n$ is $M$-separable for all $n\in\w$, so it remains to 
apply Corollary~\ref{cor_sfinomom_ufinomon}.
\hfill $\Box$
\medskip

%%%%%%%%%%%%%%%%%%%%%%%%%%%%%%%%%%%%%%%%%%%%

One of the main steps in the proof of Theorem~\ref{fu_vs_s} presented below, namely Claim~\ref{cl_0001}, is rather standard, e.g., a similar argument appears in the proof of \cite[Proposition~4.1]{GruSak11}. We nonetheless decided 
to include the proof  of Claim~\ref{cl_0001} for the sake of completeness. 
\medskip

\noindent\textit{Proof of Theorem \ref{fu_vs_s}.} \ 
Let $\la D_n:n\in\w\ra$ be a sequence of dense subsets of a FU space $X$.
Let $\{x_k:k\in\w\}$ be an enumeration of $X$ and 
 $K=\{k\in\w:x_k$ is not isolated$\}$. 
If $K=\emptyset$ then $X$ is a countable discrete space and such spaces
are clearly $S$-separable, even $H$-separable.
So let us assume that $K\neq\emptyset$.
Let us note that $X\times \w$ is also a countable FU space, where 
$\w$ is equipped with the discrete topology. Also, if
$X\times \w$ is $S$-separable then so is $X$ because the combinatorial separability properties we consider are easily seen to be preserved by 
open subspaces. Thus, replacing $X$ with $X\times\w$, if necessary,
we may (and will) assume that $K$ is infinite.
 
 \begin{claim} \label{cl_0001}
For every $N\in [\w]^\w$ and $k\in K$ there exists
a strictly increasing sequence $\la n_i:i\in\w\ra\in N^\w$
and an injective sequence
$\la y_i:i\in\w \ra$ convergent to $x_k$
such that $y_i \in D_{n_i}$ 
 for every $i\in\w$.
 \end{claim}
 \begin{proof}
Given $k\in K$,
let $\la z_n:n\in N\ra \in (X\setminus\{x_k\})^\w$ be a sequence 
convergent to $x_k$.
For every  $n\in N$  let $\la y^{n}_{j}:j\in\w\ra$ be a sequence 
of elements of $D_n\setminus\{x_k\} $ converging to $z_n$.
Thus  
$$ x_k\in\overline{\{y^{n}_j:n\in N,j\in\w\}}\setminus \{y^{n}_j:n\in N,j\in\w\},$$ 
and hence there exists an injective sequence $\la y_i:i\in\w\ra$ of elements
of $\{y^{n}_j:n\in N,j\in\w\}$ convergent to $x_k$. 
Clearly, {since $X$ is Hausdorff,}
$$|\{y_i:i\in\w\}\cap\{y^n_j:j\in\w\}|<\w$$
for all $n\in N$ since $\la y_i:i\in\w\ra$ and $\la y^n_j:j\in\w\ra$
have different limit points. Thus, passing to a subsequence
of $\la y_i:i\in\w\ra$, if necessary, we may assume that
$y_i\in D_{n_i}$ for all $i\in\w$, where  $\la n_i:i\in\w\ra\in N^\w$
is strictly increasing.
 \end{proof}
 Using Claim~\ref{cl_0001} by induction  over $k\in K$ we can construct a decreasing sequence
 $\la N_k :k\in K\ra$ of infinite subsets of $\w$, and for every
 $k\in K$ 
a strictly increasing sequence $\la n^k_i:i\in\w\ra\in N_k^\w$,
and an injective sequence
$\la y^k_i:i\in\w \ra$ convergent to $x_k$
such that $y^k_i \in D_{n^k_i}$ 
 for every $i\in\w$, making sure that 
 $N_{\mathrm{next}_K(k)} =\{n^k_i:i\in\w\}$
 for every $k\in K$. (Here $\mathrm{next}_K(k)=\min(K\setminus (k+1))$
 for every $k\in K$.)
 
 Let $\la m_k:k\in K\ra$ be a strictly increasing sequence such that 
 $m_k\in N_k$.
 For every $k\in K\setminus\{\min(K)\}$ and $l\in K\cap k$ find 
 $i(k,l)$ such that $m_k=n^l_{i(k,l)}$
 and set 
 $$F^0_{m_k} = \{y^l_{i(k,l)}:l\in K\cap k\}\in [D_{m_k}]^{<\w}. $$
 If $n\not\in\{m_k:k\in K\}$, we set $F^0_n=\emptyset$. 
 Let also $F^1_n=\{x_l:l\in\w\setminus K,l\leq n\}$
 and note that $F^1_n\subset D_n $
 because each dense set must contain all isolated points.
 Finally, we claim that the sets $F_n:=F^0_n\cup F^1_n$, $n\in\w$,
 are as required. Indeed,
 if $x_l$ is  isolated, then
 $x_l\in F^1_n\subset F_n$ for all but finitely many
 $n$. In particular,  $\{x_l\}\cap F_{m_k}\neq\emptyset$
 for all but finitely many $k\in K$.

Let now $O$ be an open neighborhood of some $x_l,$ $l\in K$.
Since $\la y^l_i:i\in\w \ra$ converges to $x_l$,
we have that 
$y^l_i\in O$ for all but finitely many $i\in\w$. In particular,
$O$ contains $y^l_{i(k,l)}$ for all but finitely many 
$k\in K\setminus (l+1)$, and 
$y^l_{i(k,l)}\in F^0_{m_k}\subset F_{m_k}$ for all $k\in K\setminus(l+1)$.
Thus, $O\cap F_{m_k}\neq\emptyset$ for all but finitely many $k\in K$.
All in all, for every open non-empty
$U\subset X$, $U\cap F_{m_{k}}\neq\emptyset$ for all but finitely many
$k\in K$, which clearly yields the $S$-separability of $X$.
\hfill $\Box$
\medskip

\noindent \textbf{Acknowledgments.} The second author would like to thank the “National Group for Algebraic and Geometric Structures, and their Applications” (GNSAGA– INdAM) for generous support for this research. A part of the results presented in this paper was obtained in July 2024 when the third named author visited the first named one at the University of S$\tilde{\mathrm{a}}$o Paulo in S$\tilde{\mathrm{a}}$o Carlos. This visit was supported by FAPESP grant 2023/00595-6. The third named author would like to thank the first one and his institute members for their great hospitality.\\

The authors express their gratitude to the referee for the thorough reading of the manuscript and the suggested corrections and improvements.

\end{document}